\documentclass[11pt,reqno]{amsart}
\usepackage{setspace}
\usepackage{amsmath, amssymb, amscd, amsthm, amsfonts}
\usepackage{graphicx}
\usepackage[pagebackref]{hyperref}
\usepackage{enumitem}
\usepackage{mathdots}

\usepackage{indentfirst}
\usepackage{amssymb,amsmath,amsfonts,amsthm}
\usepackage[all]{xy}
\usepackage{enumerate}
\usepackage{mathrsfs}
\usepackage{graphicx}
\numberwithin{equation}{section}
\usepackage[all]{xy}
\usepackage{amssymb,amscd,amsthm,amsmath,graphicx,color}
\usepackage{mathrsfs}
\usepackage[english]{babel}
\usepackage[utf8x]{inputenc}
\usepackage{hyperref}

\newtheorem{theorem}{Theorem}[section]
\newtheorem*{proposition*}{Proposition}
\newtheorem{proposition}[theorem]{Proposition}
\newtheorem{corollary}[theorem]{Corollary}
\newtheorem{lemma}[theorem]{Lemma}

\newtheorem{remark}[theorem]{Remark}

\newtheorem{question}[theorem]{Question}

\DeclareMathOperator{\Hom}{Hom}
\DeclareMathOperator{\Ext}{Ext}
\DeclareMathOperator{\Tor}{Tor}

\DeclareMathOperator{\Supp}{Supp}
\DeclareMathOperator{\depth}{depth}

\DeclareMathOperator{\pd}{pd}

\DeclareMathOperator{\id}{id}

\DeclareMathOperator{\type}{r}
\DeclareMathOperator{\D}{D}

\title{Bass and Betti numbers of a module and its deficiency modules}

\author{Thiago Fiel}
\address{Departamento de Matemática, Universidade Federal da Paraíba - 58051-900, João Pessoa, PB, Brazil}
\email{thiagofieldacostacabral@gmail.com}

\author{Rafael Holanda}
\email{rfh@academico.ufpb.br, rf.holanda@gmail.com}

\begin{document}

\maketitle

{\let\thefootnote\relax\footnote{{{\it Date:} \today}}}
{\let\thefootnote\relax}
{\let\thefootnote\relax}
{\let\thefootnote\relax\footnote{{The second-named author was supported by a CAPES Doctoral Scholarship.}}}

\begin{abstract}
This paper aims to provide several relations between Bass and Betti numbers of a given module and its deficiency modules. Such relations and the tools used throughout allow us to generalize some results of Foxby, characterize Cohen-Macaulay modules in equidimensionality terms, study the Cohen-Macaulay and complete intersection properties of a ring and furnish a case for the Auslander-Reiten conjecture.
\end{abstract}

\section{Introduction}

In the celebrated paper \cite{F}, Foxby proved that over a Gorenstein local ring $R$ of dimension $d$, a Cohen-Macaulay $R$-module $M$ of dimension $t$ is such that $$\beta_j(M)=\mu^{j+t}(\Ext^{d-t}_R(M,R))$$ and $$\mu^j(M)=\beta_{j-t}(\Ext^{d-t}_R(M,R))$$ for all $j\geq0$. In particular, $\pd_RM<\infty$ if and only if $\id_R\Ext^{d-t}_R(M,R)<\infty$, and $\id_RM<\infty$ if and only if $\pd_R\Ext^{d-t}_R(M,R)<\infty$. Recently, Freitas and Jorge-Pérez \cite{FJP} generalized the first equivalence for local rings which are factor of Gorenstein local rings. In this paper, we shall look at these results in a wider situation as follows.

Schenzel \cite{S} generalized the notion of canonical module in the following sense. Given a Noetherian local ring $R$ which is a factor ring of a $s$-dimensional Gorenstein local ring $S$ and a finite $R$-module $M$, the \emph{$j$-th deficiency module of $M$} is defined as $$K^j(M)=\Ext^{s-j}_S(M,S)$$ for all $j=0,...,\dim_RM$. Local duality assures that these modules are well-defined. Particularly, $K(M):=K^{\dim_RM}(M)$ is called the \emph{canonical module of $M$}. In a certain sense, the deficiency modules of $M$ measure the extent of the failure of $M$ to be Cohen-Macaulay.

In this paper, we shall look for relations between Bass and Betti numbers of a given module and its deficiency modules. As Foxby provided the relations above for Cohen-Macaulay modules over a Gorenstein local ring, we furnish the same relations for generalized Cohen-Macaulay canonically Cohen-Macaulay modules with zeroth and first deficiency modules of positive depth over a local ring which is a factor of a Gorenstein local ring, see Theorem \ref{foxbygeneralization}. Furthermore, the theorems \ref{foxbygeneralization2} and \ref{foxbygeneralization3} show the same relations for arbitrary finite $R$-modules when certain homological conditions over its deficiency modules are imposed.

Besides such generalizations, we exhibit bounds for the Bass numbers (Betti numbers) of a module in terms of the Betti numbers (Bass numbers) of its deficiency modules, see the theorems \ref{mu<beta} and \ref{beta<mu}. They provide several applications that are worked out through this paper. Three examples of such applications are Corollary \ref{bassgeneralization}, providing the Cohen-Macaulay property of a local ring in terms of homological conditions over deficiency modules, Corollary \ref{CIchar} furnishing a characterization of the complete intersection property in terms of the first and second Bass numbers of the residue field, and Corollary \ref{AR} that states that the Auslander-Reiten conjecture holds for modules such that its deficiency modules have finite injective dimensions, generalizing then a similar application given quite recently in \cite{FJP}.

Our methods are especially concerned with studying the behaviour of some spectral sequences. The first of them is called Foxby spectral sequence \ref{foxbyss}, as it was firstly used by Foxby in \cite{F}. The first applications of such spectral sequences regard general information on the canonical module of a generalized Cohen-Macaulay module or an equidimensional module, see Theorem \ref{GCMtheorem} and Proposition \ref{equidimensional}. These results provide sufficient conditions for when the module is also canonically Cohen-Macaulay and its canonical module is generalized Cohen-Macaulay, see the corollaries \ref{GCMCCM} and \ref{GCMcanonicalmodule}, also a characterization of Cohen-Macaulay modules in Corollary \ref{CMequivalence} and a version for generalized Cohen-Macaulay modules of a Schenzel's result, see Corollary \ref{weakercmcanonicalmodule}.

\section{Generalized Cohen-Macaulay modules}

\noindent\textbf{Setup.} Throughout this paper, $R$ will always denote a commutative Noetherian local ring with non-zero unity, maximal ideal $\mathfrak{m}$ and residue class field $k$. Also, $R$ is supposed to be a factor of a Gorenstein local ring $S$ of dimension $s$, i.e., there exists a surjective ring homomorphism $S\rightarrow R$. We say that an $R$-module $M$ is \emph{finite} if it is a finitely generated $R$-module and denote by $M^\vee$ its Matlis dual.

For an $R$-module $M$, $\pd_RM$ and $\id_RM$ denote, respectively, the projective dimension and injective dimension of $M$. Further, $\beta_i(M)=\dim_k\Tor^R_i(k,M)$ is the $i$-th Betti number of $M$, $\mu^i(M)=\dim_k\Ext^i_R(k,M)$ is the $i$-th Bass number of $M$ and $\type(M)=\dim_k\Ext^{\depth_RN}_R(k,M)$ is its type.

The following spectral sequences have first appeared in the \cite{F}.

\begin{lemma}[Foxby spectral sequences]\label{foxbyss}
Given a finite $R$-module $X$, an $R$-module $Y$ and a $S$-module $Z$, if either $\pd_RX<\infty$ or $\id_SZ<\infty$, then there exist a graded $R$-module $H$ and first quadrant spectral sequences
$$E_2^{p,q}=\Ext^p_S(\Ext^q_R(X,Y),Z)\Rightarrow_p H^{q-p}$$ and
$$'E_2^{p,q}=\Tor_R^p(X,\Ext^q_S(Y,Z))\Rightarrow_p H^{p-q}.$$
\end{lemma}

\begin{proof}
Let $F_\bullet$ be a free $R$-resolution of $X$ and let $E^\bullet$ be an injective $S$-resolution of $Z$. The desired spectral sequences yield from the isomorphism of first quadrant double complexes
$$\Hom_S(\Hom_R(F_\bullet,Y),E^\bullet)\simeq F_\bullet\otimes_R\Hom_S(Y,E^\bullet).$$
\end{proof}

The first application of the Foxby spectral sequences \ref{foxbyss} is a generalization of a well-known result about Cohen-Macaulay modules and its canonical modules, see \cite[Theorem 1.14]{S}. First, we need an auxiliary lemma.

We say that a finite $R$-module $M$ satisfies \emph{Serre's condition $S_k$}, for $k$ being a non-negative integer, provided $$\depth_{R_\mathfrak{p}}M_\mathfrak{p}\geq\min\{k,\dim_{R_\mathfrak{p}}M_{\mathfrak{p}}\}$$ for all $\mathfrak{p}\in\Supp M$.

\begin{lemma}\cite[Lemma 1.9]{S}\label{schenzellemma}
Let $M$ be a finite $R$-module of dimension $t$. The modules $K^j(M)$ satisfy the following properties.
\begin{itemize}
    \item [(i)] $\dim_R K^j(M)\leq j$ for all integer $j$ and $\dim_RK(M)=t$;
    \item [(ii)] Suppose that $M$ is equidimensional. Then, $M$ satisfies Serre's condition $S_k$ if and only if $\dim_RK^j(M)\leq j-k$, for all $0\leq j<t$.
\end{itemize}
\end{lemma}

A finite $R$-module $M$ is said to be \emph{generalized Cohen-Macaulay} if $H^j_\mathfrak{m}(M)$ is of finite length for all $j<\dim_RM$. It should be noticed, due to Matlis duality, that it is equivalent to say that $K^j(M)$ is of finite length for all $j<\dim_RM.$

\begin{theorem}\label{GCMtheorem}
Let $M$ be a generalized Cohen-Macaulay $R$-module of dimension $t$. The following statements hold.
\begin{itemize}
    \item [(i)] There exists isomorphism $$K^0(K(M))\simeq\Tor_{-t}^S(M,S);$$
    \item [(ii)] There exists a five-term type exact sequence
    $$\xymatrix@=1em{
    \Tor^S_{-t+2}(M,S)\ar[r] & K^2(K(M))\ar[r] & K^0(K^{t-1}(M))\ar[dl] \\ & \Tor^S_{-t+1}(M,S)\ar[r] & K^1(K(M))\ar[r] & 0
    }$$
    \item [(iii)] There exists an exact sequence
    $$\xymatrix@=1em{
    0\ar[r] & K^0(K^0(M))\ar[r] & M\ar[r] & K(K(M))\ar[r] & K^0(K^1(M))\ar[r] & 0;
    }$$
    \item [(iv)] If $t\geq3$, then there exist isomorphisms $$K^{t-j}(K(M))\simeq K^0(K^{j+1}(M))$$ for all $1\leq j\leq t-2$.
\end{itemize}
\end{theorem}

\begin{proof}
Consider the Foxby spectral sequences \ref{foxbyss} by taking $X=M$ as $S$-module and $Y=Z=S$

$$E_2^{p,q}=\Ext^p_S(\Ext^q_S(M,S),S)\Rightarrow_p H^{q-p}$$ and
$$'E_2^{p,q}=\Tor^S_p(M,\Ext^q_S(S,S))\Rightarrow_p H^{p-q}.$$

Since $'E_2^{p,q}=0$ for all $q\neq0$, we have
$$H^j\simeq{}'E_2^{j,0}=\Tor_j^S(M,S)$$ for all $j\geq0$, and $$E_2^{p,q}=\Ext^p_S(\Ext^q_S(M,S),S)\Rightarrow_p\Tor^S_{q-p}(M,S).$$

Once $H^j_\mathfrak{m}(M)$ being of finite length, so is $K^j(M)$ for all $j<t$, and by local duality $$\Ext^p_S(\Ext^q_S(M,S),S)=\Ext^p_S(K^{s-q}(M),S)=0$$ for all $q>s-t$ and for all $p\neq s$. Also, Lemma \ref{schenzellemma} $(i)$ assures that $\dim_RK(M)=t$. Thus, $E_2$ has the following shape

$$
\xymatrix@=1em{
0 & 0 & 0 & \cdots & 0 & 0 \\
0 & 0 & 0 & \cdots & \Ext^s_S(K^0(M),S) & 0 \\
\vdots & \vdots & \vdots & \iddots & \vdots & \vdots\\
0 & 0 & 0  & \cdots & \Ext^s_S(K^{t-1}(M),S)  & 0 \\
0 & K(K(M)) & \Ext^{s-(t-1)}_S(K(M),S)  & \cdots &  \Ext^s_S(K(M),S)  & 0 \\
0\ar@{--}[rrrrruuuuu] & 0 & 0 & 0 & 0 & 0.
}
$$

By convergence, there are isomorphisms
$$K^0(K(M))=\Ext^s_S(K(M),S)\simeq E_\infty^{s,s-t}\simeq\Tor_{-t}^S(M,S), \ K^1(K(M))=\Ext^{s-1}_S(K(M),S)\simeq E_\infty^{s-1,s-t}$$ and
$$K^0(K^0(M))=\Ext^s_S(K^0(M),S)\simeq E_\infty^{s,s}.$$
Thus we get item $(i)$ and by applying Matlis dual one has isomorphisms
$$H^1_\mathfrak{m}(K(M))\simeq(E_\infty^{s-1,s-t})^\vee \ \mbox{and} \ H^0_\mathfrak{m}(K^0(M))\simeq(E^{s,s}_\infty)^\vee.$$
The convergence again gives us short exact sequences
\begin{equation}\xymatrix@=1em{
0\ar[r] & E_\infty^{s,s-j}\ar[r] & \Tor_{-j}^S(M,S)\ar[r] & E_\infty^{s-(t-j),s-t}\ar[r] & 0
}\label{eq:GCMconv}\end{equation} for all $j\geq0$.
Further, as we move through the pages of $E$, the differentials between the vertical and horizontal lines in the diagram above come out. In other words, there is an exact sequence
\begin{equation}\xymatrix@=1em{
0\ar[r] & E_\infty^{s-(t-j),s-t}\ar[r] & \Ext_S^{s-(t-j)}(K(M),S)\ar[r] & \Ext^s_S(K^{j+1}(M),S)\ar[r] & E_\infty^{s,s-(j+1)}\ar[r] & 0
}\label{eq:GCMdifferentials}\end{equation}
for all $0\leq j\leq t-2$.

Item $(ii)$ is exactly the five-term exact sequence of $E$. For item $(iii)$, by taking $j=0$ in both above exact sequences, we have the following exact sequences
$$\xymatrix@=1em{
0\ar[r] & \Ext^s_S(K^0(M),S)\ar[r] & M\ar[r] & E^{s-t,s-t}_\infty\ar[r] & 0
}$$ and
$$\xymatrix@=1em{
0\ar[r] & E^{s-t,s-t}_\infty\ar[r] & K(K(M))\ar[r] & \Ext^s_S(K^1(M),S)\ar[r] & E_\infty^{s,s-1}\ar[r] & 0.
}$$ The result follows by splicing these sequences and noticing that $E_\infty^{s,s-1}\subseteq\Tor^S_{-1}(M,S)=0$.

The exact sequence \ref{eq:GCMconv} assures that $E_\infty^{s-(t-j),s-t}=E_\infty^{s,s-j}=0$ for all $j>0$, so that, by the exact sequence \ref{eq:GCMdifferentials}, $$K^{t-j}(K(M))=\Ext^{s-(t-j)}_S(K(M),S)\simeq\Ext^s_S(K^{j+1}(M),S)=K^0(K^{j+1}(M))$$ for all $1\leq j\leq t-2$.
\end{proof}

The concept of \emph{canonically Cohen-Macaulay module} was introduced by Schenzel \cite{S2}. We say that a finite $R$-module $M$ is canonically Cohen-Macaulay if its canonical module $K(M)$ is Cohen-Macaulay.

\begin{corollary}\label{GCMCCM}
Let $M$ be a generalized Cohen-Macaulay $R$-module of dimension $t$. The following statements hold.
\begin{itemize}
    \item [(i)] If $t>j$ with $j\in\{0,1\}$, then $\depth_RK(M)>j$;
    \item [(ii)] If $t=1$, then $M$ is canonically Cohen-Macaulay and there exists the short exact sequence 
    $$\xymatrix@=1em{
    0\ar[r] & K^0(K^0(M))\ar[r] & M\ar[r] & K(K(M))\ar[r] &  0;
    }$$
    \item [(iii)] If $t=2$, then $M$ is canonically Cohen-Macaulay;
    \item [(iv)] If $t\geq3$, then $K(M)$ is generalized Cohen-Macaulay.
\end{itemize}
\end{corollary}

\begin{proof}
Item $(i)$ follows immediately from Theorem \ref{GCMtheorem} $(i)$ and $(ii)$. For item $(ii)$, item $(i)$ assures that $K(M)$ is Cohen-Macaulay and Theorem \ref{GCMtheorem} $(iii)$ is the desired exact sequence. As to item $(iii)$, item $(i)$ again assures that $K(M)$ is Cohen-Macaulay. Item $(iv)$ follows directly from item $(i)$ and Theorem \ref{GCMtheorem} $(iv)$.
\end{proof}

\begin{corollary}\label{GCMcanonicalmodule}
If $M$ is generalized Cohen-Macaulay, then so is $K(M)$.
\end{corollary}

Corollary \ref{GCMcanonicalmodule} inspires us to ask the following.

\begin{question}
Given a finite $R$-module $M$, when is $K(M)$ generalized Cohen-Macaulay? 
\end{question}

As Corollary \ref{GCMCCM} assures that generalized Cohen-Macaulay of dimension at most two are canonically Cohen-Macaulay, Theorem \ref{GCMtheorem} $(iv)$ recovers a characterization \cite{BN} for the case where the dimension is at least three.

\begin{corollary}\cite[Corollary 2.7]{BN}
Let $M$ be a generalized Cohen-Macaulay $R$-module of dimension $t\geq3$. Then, the following statements are equivalent
\begin{itemize}
    \item [(i)] $M$ is canonically Cohen-Macaulay;
    \item [(ii)] $H^j_\mathfrak{m}(M)=0$ for all $j=2,...,t-1$;
    \item [(iii)] The $\mathfrak{m}$-transform functor $\D_\mathfrak{m}(M)$ is a Cohen-Macaulay $R$-module.
\end{itemize}
\end{corollary}

\begin{proposition}\label{equidimensional}
Let $M$ be a finite $R$-module of depth $g$ and dimension $t$. The following statements hold.
\begin{itemize}
    \item [(i)] Assume $M$ is generalized Cohen-Macaulay $R$-module. If $\depth_RK^j(M)>0$ for $j=0,1$, then $M\simeq K(K(M))$. In particular, this isomorphism holds true whenever $g\geq2$.
    \item [(ii)] Suppose $M$ is equidimensional. If $M$ satisfies Serre's condition $S_{k+1}$ for some positive integer $k$, then $$K^j(K(M))\simeq \Tor^S_{-t+j}(M,S)$$ for all $t-k+1\leq j\leq t$.
\end{itemize}
\end{proposition}

\begin{proof}
Item $(i)$ follows immediately from Theorem \ref{GCMtheorem} $(iii)$ and from the fact that $K^0(M)=K^1(M)=0$ in case of $g\geq2$.

For item $(ii)$, consider the Foxby spectral sequences given in Theorem \ref{GCMtheorem} $$E_2^{p,q}=\Ext^p_S(\Ext^q_S(M,S),S)\Rightarrow_p\Tor^S_{q-p}(M,S).$$

By Lemma \ref{schenzellemma} $(ii)$ and local duality, we have $$E_2^{s-i,s-j}=\Ext^{s-i}_S(K^j(M),S)=0$$ for all $0\leq j<t$ and $i>j-k-1$. In other words, all modules $E_2^{p,q}$ such that $q\neq s-t$ above the dotted line in the below diagram must be zero

$$
\xymatrix@=1em{
0 & 0 & \cdots & 0 & 0\\
\vdots & \vdots & \iddots & \Ext^s_S(K^{k+1}(M),S) & 0 \\
0 & \Ext^{s-(t-k-2)}_S(K^{t-1}(M),S) & \cdots & \vdots & 0 \\
\Ext^{s-(t-k-1)}_S(K(M),S)\ar@{--}[rrrruuu] & \Ext^{s-(t-k-2)}_S(K(M),S) & \cdots &  \Ext^s_S(K(M),S) & 0 \\
0 & \cdots & 0 & 0 & 0.
}
$$
The result follows from the convergence.
\end{proof}

Our results also retrieve the well-known fact that every Cohen-Macaulay module is canonically Cohen-Macaulay, see \cite[Theorem 1.14]{S}.

\begin{corollary}\label{cmcanonicalmodule}
If $M$ is Cohen-Macaulay of dimension $t$, then so is $K(M)$ and $K(K(M))\simeq M$.
\end{corollary}

\begin{proof}
There are two immediate ways of proving the desired result. Indeed the result follows directly from Theorem \ref{GCMtheorem} as well as from Proposition \ref{equidimensional} $(ii)$ too.
\end{proof}

Proposition \ref{equidimensional} provides a characterization for the Cohen-Macaulay property.

\begin{corollary}\label{CMequivalence}
If $M$ is a finite $R$-module, then $M$ is Cohen-Macaulay if and only if $M$ is equidimensional canonically Cohen-Macaulay satisfying Serre's condition $S_{k+1}$ for some positive integer $k$.
\end{corollary}

\begin{proof}
It is well-known that a Cohen-Macaulay module is equidimensional and satisfies Serre's condition $S_k$ for any $k$. Corollary \ref{cmcanonicalmodule} assures that such a module is also canonically Cohen-Macaulay. Conversely, by taking $j=t$ in Proposition \ref{equidimensional} $(ii)$, we have the isomorphism $K(K(M))\simeq M$. Since $K(M)$ is Cohen-Macaulay, Corollary \ref{cmcanonicalmodule} again assures that $M\simeq K(K(M))$ is Cohen-Macaulay.
\end{proof}

The next corollary is a version of Corollary \ref{cmcanonicalmodule} for generalized Cohen-Macaulay modules.

\begin{corollary}\label{weakercmcanonicalmodule}
If $M$ is a generalized Cohen-Macaulay module such that $\depth_RK^j(M)>0$ for $j=0,1$, then so is $K(M)$ and $M\simeq K(K(M))$.
\end{corollary}

\begin{proof}
It follows directly from Corollary \ref{GCMcanonicalmodule} and Proposition \ref{equidimensional} $(i)$.
\end{proof}







\section{Bounding Bass numbers}

The Foxby spectral sequences \ref{foxbyss} are fundamental tools in our work. They provide the main result of this section.

\begin{theorem}\label{mu<beta}
If $M$ is a finite $R$-module of depth $g$ and dimension $t$, then the following inequality holds for all $j\geq0$

$$\mu^j(M)\leq\sum_{i=g}^t\beta_{j+i}(K^i(M)).$$
Moreover, $\type(M)=\beta_0(K^g(M))$ and $$\mu^{g+2}(M)-\mu^{g+1}(M)\leq\beta_2(K^g(M))-\beta_1(K^g(M))-\beta_0(K^{g+1}(M)).$$
\end{theorem}

\begin{proof}
Consider the Foxby spectral sequences \ref{foxbyss} by taking $S=R, X=k, Y=M, Z=S$.

$$E_2^{p,q}=\Ext^p_S(\Ext^q_R(k,M),S)\Rightarrow_p H^{q-p}$$ and
$$'E_2^{p,q}=\Tor_p^R(k,\Ext^q_S(M,S))\Rightarrow_p H^{p-q}.$$

Since $\Ext^q_R(k,M)$ is of finite length, we must have $E_2^{p,q}=0$ for all $p\neq s$, so that 
$$H^j\simeq E_2^{s,j+s}=\Ext^s_S(\Ext^{j+s}_R(k,M),S)$$ for all integer $j$. Once $K^{s-q}(M)=\Ext^q_S(M,S)$ for all $q\geq0$, we conclude that
\begin{equation}'E_2^{p,q}=\Tor_p^R(k,K^{s-q}(M))\Rightarrow_p\Ext^s_S(\Ext^{p-q+s}_R(k,M),S).\label{eq:1}\end{equation}

Now, since $\Ext^s_S(k,S)^\vee\simeq k$, where $\_^\vee$ denotes the Matlis dual of $R$, we have
$$\Ext^s_S(\Ext^j_R(k,M),S)\simeq\Ext^s_S(k,S)^{\mu^j(M)}\simeq k^{\mu^j(M)}$$ as $k$-vector spaces. Therefore, by the convergence of $'E$,
$$\mu^j(M)\leq\sum_{j=p-q+s}\beta_p(K^{s-q}(M))=\sum_{i=g}^t\beta_{j+i}(K^i(M))$$ for all $j\geq0$.

Now, since $K^i(M)=\Ext^{s-i}_S(M,S)=0$ for all $i<g$, then $'E_2$ has the following corner $$\xymatrix@=1em{
& \vdots &  \vdots & \vdots
\\
\cdots & \Tor_2^R(k,K^{g+1}(M)) & \Tor_2^R(k,K^g(M))\ar[ddl] & 0 & \cdots
\\
\cdots & \Tor_1^R(k,K^{g+1}(M)) & \Tor_1^R(k,K^g(M)) & 0 & \cdots
\\
\cdots & k\otimes_RK^{g+1}(M) & k\otimes_RK^g(M) & 0 & \cdots
}$$
Therefore, $$k\otimes_RK^g(M)={}'E_2^{0,s-g}\simeq H^{g-s}\simeq\Ext_S^s(\Ext^g_R(k,M),S)$$ so that $\type(M)=\beta_0(K^g(M))$ and there exists a five-term-type exact sequence
$$\xymatrix@=1em{
\Ext^s_S(\Ext^{g+2}_R(k,M),S)\ar[r] & \Tor_2^R(k,K^g(M))\ar[r] & k\otimes_RK^{g+1}(M)\ar[dl]
\\
& \Ext^s_S(\Ext^{g+1}_R(k,M),S)\ar[r] & \Tor_1^R(k,K^g(M))\ar[r] & 0
}$$ whence the desired formula.
\end{proof}

\begin{corollary}\label{finiteid}
Let $M$ be a finite $R$-module of depth $g$ and dimension $t$. If $\pd_RK^i(M)<\infty$ for all $i=g,...,t$, then $\id_RM<\infty$.
\end{corollary}

\begin{proof}
The hypothesis means that $\beta_l(K^i(M))=0$ for all $l\gg0$ and by Theorem \ref{mu<beta} one has
$$\mu^j(M)\leq\sum_{i=g}^t\beta_{j+i}(K^i(M))=0$$ for $j\gg0$, i.e., $\id_RM<\infty$.
\end{proof}

Bass' conjecture \cite{B} was first proved by Peskine-Szpiro in \cite{PS} and after in a more general setting by Roberts \cite{R}. It states that a local ring admitting a non-zero module of finite injective dimension must be Cohen-Macaulay. The next corollary provides sufficient conditions in terms of projective dimension for a local ring to be Cohen-Macaulay.

\begin{corollary}\label{bassgeneralization}
Let $M$ be a finite $R$-module of depth $g$ and dimension $t$. If $\pd_RK^i(M)<\infty$ for all $i=g,...,t$, then $R$ is Cohen-Macaulay.
\end{corollary}

\begin{proof}
Corollary \ref{finiteid} assures that $\id_RM<\infty$ and thus the result follows from Bass' conjecture.
\end{proof}

\begin{theorem}\label{foxbygeneralization}
If $M$ is a generically Cohen-Macaulay canonically Cohen-Macaulay $R$-module of dimension $t$ such that $\depth_RK^j(M)>0$ for $j=0,1$, then 
$$\beta_j(M)=\mu^{j+t}(K(M))$$ and $$\mu^j(M)=\beta_{j-t}(K(M))$$ for all $j\geq0$. In particular, $\pd_RM<\infty$ if and only if $\id_RK(M)<\infty$ and $\id_RM<\infty$ if and only if $\pd_RK(M)<\infty$.
\end{theorem}

\begin{proof}
By Lemma \ref{schenzellemma} $(i)$, $K(M)$ is Cohen-Macaulay of dimension $t$ and by Proposition \ref{equidimensional} $(i)$, $K(K(M))\simeq M$, that is, $K^i(K(M))=0$ for all $i\neq t$ and $K^t(K(M))\simeq M$. The spectral sequence \ref{eq:1}
$$'E_2^{p,q}=\Tor^R_p(k,K^{s-q}(K(M)))\Rightarrow_p\Ext^s_S(\Ext^{p-q+s}_R(k,K(M)),S)$$ degenerates, so that
$$\Tor_j^R(k,M)\simeq\Tor_j^R(k,K(K(M)))={}'E_2^{j,s-t}\simeq\Ext^s_S(\Ext^{j+t}_R(k,K(M)),S)$$ for all $j\geq0$. Therefore,
$$\beta_j(M)=\dim_k\Tor_j^R(k,M)=\dim_k\Ext^s_S(\Ext^{j+t}_R(k,K(M)),S)=\mu^{j+t}(K(M))$$ for all $j\geq0$. The other equality follows from the fact $K(K(M))\simeq M$.
\end{proof}

Theorem \ref{foxbygeneralization} generalizes \cite[Corollary 3.6]{F} and improves \cite[Corollary 3.3]{FJP}. We record this in the next corollary.

\begin{corollary}\label{foxbyresult}
If $M$ is Cohen-Macaulay $R$-module of dimension $t$, then $$\beta_j(M)=\mu^{j+t}(K(M))$$ and $$\mu^j(M)=\beta_{j-t}(K(M))$$ for all $j\geq0$. In particular, $\pd_RM<\infty$ if and only if $\id_RK(M)<\infty$ and $\id_RM<\infty$ if and only if $\pd_RK(M)<\infty$.
\end{corollary}

\begin{proof}
If $t\geq2$, then the result follows from Theorem \ref{foxbygeneralization}. Otherwise, Corollary \ref{cmcanonicalmodule} and the spectral sequence argument given in the proof of Theorem \ref{foxbygeneralization} asserts the result.
\end{proof}




The next theorem is an attempt to extent part of Theorem \ref{foxbygeneralization} to arbitrary modules. In the next section, we work on the other part.

\begin{theorem}\label{foxbygeneralization2}
Let $M$ be a finite $R$-module of depth $g$ and dimension $t$. If $\pd_RK^i(M)<\infty$ for all $g\leq i<t$, then $$\mu^j(M)=\beta_{j-t}(K(M))$$ for all $j>\depth R+t$. In particular, $\id_RM<\infty$ if and only if $\pd_RK(M)<\infty$.
\end{theorem}

\begin{proof}
The spectral sequence \ref{eq:1} is such that $'E_2^{p,q}=0$ for all $p>\depth R$ and $g\leq q<t$, so that $$\Tor_j^R(k,K(M))={}'E_2^{j,s-t}\simeq\Ext^s_S(\Ext^{j+t}_R(k,M),S),$$ whence the result.
\end{proof}

We derive other consequences of Theorem \ref{mu<beta}. In particular, we say exactly when the type of a finite module is one in terms of its deficiency modules.


\begin{corollary}\label{typecaracterization}
Let $M$ be a finite $R$-module of depth $g$ and dimension $t$. The following statements hold.
\begin{itemize}
    \item [(i)] If $M$ is Cohen-Macaulay of dimension $t$, then $$\mu^{t+2}(K(M))-\mu^{t+1}(K(M))\geq\beta_2(M)-\beta_1(M).$$ In particular, if $\pd_RM<\infty$ then $\beta_1(M)\geq\beta_2(M)$.
    \item [(ii)] If $\id_RM<\infty$, then $$\beta_0(K^{g+1}(M))\geq\beta_2(K^g(M))-\beta_1(K^g(M)).$$ In particular, if $M$ is also Cohen-Macaulay, then $\beta_1(K(M))\geq\beta_2(K(M))$.
    \item [(iii)] $\type(M)=1$ if and only if $K^g(M)$ is cyclic.
\end{itemize}
\end{corollary}

\begin{proof}
Item $(iii)$ follows directly from Theorem \ref{mu<beta}. Item $(i)$ follows from Corollary \ref{cmcanonicalmodule}, Theorem \ref{mu<beta} and Corollary \ref{foxbyresult}, and item $(ii)$ follows from \cite[Theorem 3.7]{BH}, corollaries \ref{cmcanonicalmodule} and \ref{foxbyresult} and item $(i)$.
\end{proof}

The spectral sequence \ref{eq:1} provides more information when the module involved has only two (possibly) non-zero deficiency modules.

\begin{proposition}\label{t=g+rfiniteid}
Let $M$ be a finite $R$-module of depth $g$ and dimension $t$. Suppose $K^i(M)=0$ for all $i\neq g,t$. If $\id_RM<\infty$ then $\beta_j(K^g(M))=\beta_{j+g-t-1}(K(M))$ for all $j>\depth R-g+1$.
\end{proposition}

\begin{proof}
Write $t=g+r$. The spectral sequence \ref{eq:1} has only two vertical lines as the following diagram shows
$$\xymatrix@=1em{
& & \vdots & \vdots & & \vdots & \vdots
\\
\cdots & 0 & \Tor_{r+1}^R(k,K(M)) & 0 & \cdots & 0 & \Tor_{r+1}^R(k,K^g(M))\ar[ddddllll] & 0 & \cdots
\\
& & \vdots & & \iddots & & \vdots
\\
\cdots & 0 & \Tor_2^R(k,K(M)) & 0 & \cdots & 0 & \Tor_2^R(k,K^g(M)) & 0 & \cdots
\\
\cdots & 0 & \Tor_1^R(k,K(M)) & 0 & \cdots & 0 & \Tor_1^R(k,K^g(M)) & 0 & \cdots
\\
\cdots & 0 & k\otimes_RK(M) & 0 & \cdots & 0 &  k\otimes_RK^g(M) & 0 & \cdots
}$$

From convergence, we obtain an exact sequence
$$\xymatrix@=1em{
\Ext^s_S(\Ext^{j+g}_R(k,M),S)\ar[r] & \Tor_j^R(k,K^g(M))\ar[r] & \Tor_{p-r-1}^R(k,K(M))\ar[r] & \Ext^s_S(\Ext^{j+g-1}_R(k,M),S)}$$ for all $j\geq0$. Thus, since $\id_RM=\depth R$ (see \cite[Theorem 3.7.1]{BH}), we conclude that $$\Tor_j^R(k,K^g(M))\simeq\Tor_{j-r-1}^R(k,K(M))$$ for all $j>\depth R-g+1$, whence the result.
\end{proof}

Based on Corollary \ref{finiteid} and Proposition \ref{t=g+rfiniteid}, we finish this section by asking the following.

\begin{question}\label{question1}
Let $M$ be a finite $R$-module of depth $g$ and dimension $t$. Is it true that $$\id_RM<\infty\Leftrightarrow\pd_RK^i(M)<\infty,\forall i=g,...,t?$$
\end{question}

\section{Bounding Betti numbers}

In last section, we bounded the Bass numbers of a module in terms of the Betti numbers of the deficiency modules. In this section, we get a dual version of Theorem \ref{mu<beta} in the following sense.

\begin{theorem}\label{beta<mu}
For a finite $R$-module $M$ of depth $g$ and dimension $t$, the following inequality holds true for all $j\geq0$
$$\beta_j(M)\leq\sum_{i=g}^t\mu^{j+i}(K^i(M)).$$ Moreover, $\mu^0(K(M))=\beta_{-t}(M)$ and $$\beta_{-t+2}(M)-\beta_{-t+1}(M)\geq\mu^2(K(M))-\mu^1(K(M))-\mu^0(K^{t-1}(M)).$$
\end{theorem}

\begin{proof}
By taking a free $R$-resolution $F_\bullet$ of $k$ and an injective $S$-resolution $E^\bullet$ of $S$, the tensor-hom adjunction induces a first quadrant double complex isomorphism $$\Hom_S(F_\bullet,\Hom_S(M,E^\bullet))\simeq\Hom_S(F_\bullet\otimes_Rk,E^\bullet)$$ which yields two spectral sequences as follows
$$E_2^{p,q}=\Ext^p_R(k,\Ext^q_S(M,S))\Rightarrow_p H^{p+q}$$ 
and
$$'E_2^{p,q}=\Ext^p_S(\Tor_q^R(k,M),S)\Rightarrow_p H^{p+q}.$$ 
Since $\Tor^R_q(k,M)$ is of finite length for all $q\geq0$, due to local duality, we must have $'E_2^{p,q}=0$ for all $p\neq s$, so that
$$H^j\simeq{}'E_2^{s,j-s}=\Ext^s_S(\Tor_{j-s}^R(k,M),S)$$ 
for all $j\geq0$. Once $K^{s-q}(M)=\Ext^q_R(M,S)$ for all $q\geq0$, one has spectral sequence
\begin{equation}E_2^{p,q}=\Ext^p_R(k,K^{s-q}(M))\Rightarrow_p\Ext^s_S(\Tor^R_{p+q-s}(k,M),S).\label{eq:2}\end{equation}
By convergence, we conclude that
$$\beta_j(M)=\dim_k\Ext^ s_S(\Tor^R_{(j+s)-s}(k,M),S)\leq\sum_{p+q=j+s}\dim_k\Ext^p_R(k,K^{s-q}(M))=\sum_{i=g}^t\mu^{i+j}(K^i(M)).$$

Now, since $K^i(M)=0$ for all $i<g$ or $i>t$, then $E_2^{p,q}=0$ for all $q<s-t$ or $q>s-g$. In particular, $E_2$ has a corner as follows
$$\xymatrix@=1em{
\vdots & \vdots & \vdots
\\
\Hom_R(k,K^{t-1}(M))\ar[drr] & \Ext^1_R(k,K^{t-1}(M)) & \Ext^2_R(k,K^{t-1}(M)) & \cdots
\\
\Hom_R(k,K(M)) & \Ext^1_R(k,K(M)) & \Ext^2_R(k,K(M)) & \cdots
\\
0 & 0 & 0 & \cdots
\\
\vdots & \vdots & \vdots}$$
Therefore, there exists the  isomorphism $$\Hom_R(k,K(M))=E_2^{0,s-t}\simeq\Ext^s_S(\Tor^R_{-t}(k,M),S)$$ and a five-term type exact sequence
$$\xymatrix@=1em{
0\ar[r] & \Ext^1_R(k,K(M))\ar[r] & \Ext^s_S(\Tor^R_{-t+1}(k,M),S)\ar[r] & \Hom_R(k,K^{t-1}(M))\ar[dl]
\\
& & \Ext^2_R(k,K(M))\ar[r] & \Ext^s_S(\Tor^R_{-t+2}(k,M),S)
}$$
whence the result.
\end{proof}

\begin{remark}
It should be noticed that the estimate $\beta_j(M)\leq\sum_{i=g}^t\mu^{j+i}(K^i(M))$ is already known, see \cite[Theorem 3.2]{S}.
\end{remark}

\begin{corollary}\label{t=0}
The following statements hold.
\begin{itemize}
    \item [(i)] If $t=0$, then $\beta_0(M)=\mu^0(K(M))$ and $$\beta_2(M)-\beta_1(M)\geq\mu^2(K(M))-\mu^1(K(M)).$$ Otherwise, $\depth_RK(M)>0$;
    \item [(ii)] If $t=1$, then $\beta_1(M)-\beta_0(M)\geq\mu^2(K(M))-\mu^1(K(M))-\mu^0(K^0(M))$;
    \item [(iii)] If $t=2$, then $\beta_0(M)\geq\mu^2(K(M))-\mu^1(K(M))-\mu^0(K^1(M))$;
    \item [(iv)] If $t>2$, then $\mu^0(K^{t-1}(M))\geq\mu^2(K(M))-\mu^1(K(M))$.
\end{itemize}
\end{corollary}

\begin{proof}
It follows directly from Theorem \ref{beta<mu}.
\end{proof}

\begin{corollary}\label{artinianlemma}
If $M$ is a finite Artinian $R$-module, then $$\beta_2(M)-\beta_1(M)=\mu^2(K(M))-\mu^1(K(M)).$$
\end{corollary}

\begin{proof}
By the corollaries \ref{typecaracterization} $(i)$ and \ref{t=0} $(i)$, $$\mu^2(K(M))-\mu^1(K(M))\geq\beta_2(M)-\beta_1(M)\geq\mu^2(K(M))-\mu^1(K(M)).$$
\end{proof}

\begin{lemma}\cite[Proposition 2.8.4]{H}\label{CIlemma}
Suppose $R$ is $d$-dimensional with embedding dimension $e$. Then $\beta_1(R/\mathfrak{m})=e$ and the following statements are equivalent.
\begin{itemize}
    \item [(i)] $\beta_2(R/\mathfrak{m})=\binom{e}{2}+e-d$;
    \item [(ii)] $R$ is a complete intersection.
\end{itemize}
\end{lemma}

\begin{corollary}\label{CIchar}
If $R$ is $d$-dimensional of embedding dimension $e$, then $$\mu^2(k)-\mu^1(k)=\binom{e}{2}-d$$ if and only if $R$ is a complete intersection.
\end{corollary}

\begin{proof}
It follows directly from Corollary \ref{artinianlemma} and Lemma \ref{CIlemma}.
\end{proof}

\begin{corollary}\label{finitepd}
Let $M$ be a finite $R$-module of depth $g$ and dimension $t$. If $\id_RK^i(M)<\infty$ for all $i=g,...,t$, then $\pd_RM<\infty$.
\end{corollary}

\begin{proof}
By hypothesis, we have $\mu^l(K^i(M))=0$ for all $l\gg0$ and by Theorem \ref{beta<mu} one has
$$\beta_j(M)\leq\sum_{i=g}^t\mu^{j+i}(K^i(M))=0$$ for all $j\gg0$, whence $\mu^j(M)=0$ for all $j\gg0$, that is, $\pd_RM<\infty$.
\end{proof}

The \emph{Auslander-Reiten conjecture} \cite{AR} states the following. Given a finite $R$-module $M$, if $$\Ext^j_R(M,M\oplus R)=0$$ for all $j>0$, then $M$ is free.
This long-standing conjecture has been largely studied and several positive answers are already known, see for instance \cite{A,AY,AB,DEL,FJP,HL,LM,NS}. Corollary \ref{finitepd} provides another positive answer for the Auslander-Reiten conjecture for a class of modules. But first, we need a lemma.

\begin{lemma}\cite[Lemma 1 (iii)]{M}\label{pdfinite}
Let $R$ be a local ring and let $M$ and $N$ be finite $R$-modules. If $\pd_RM<\infty$ and $N\neq0$, then

$$\pd_RM=\sup\{j:\Ext^j_R(M,N)\neq0\}.$$
\end{lemma}

\begin{theorem}\label{arconjtheorem}
Let $M$ be a finite $R$-module of depth $g$ and dimension $t$. $M$ is free provided the following statements hold.
\begin{itemize}
    \item [(i)] $\id_RK^i(M)<\infty$ for all $i=g,...,t$;
    \item [(ii)] There exists an $R$-module $N$ such that $\Ext^j_R(M,N)=0$ for all $j=1,...,d$.
\end{itemize}
\end{theorem}

\begin{proof}
It follows directly from Corollary \ref{finitepd} and Lemma \ref{pdfinite}.
\end{proof}

The next corollary proves the Auslander-Reiten conjecture for a certain class of modules. It generalizes the case of the conjecture obtained in \cite{FJP}.

\begin{corollary}\label{AR}
The Auslander-Reiten conjecture holds for finite modules having deficiency modules of finite injective dimension over local rings which are factors of Gorenstein local rings.
\end{corollary}

\begin{proof}
It follows immediately from Theorem \ref{arconjtheorem}.
\end{proof}
In the next theorem, such as Theorem \ref{foxbygeneralization2}, we furnish another attempt to remove the generalized Cohen-Macaulayness hypothesis from Theorem \ref{foxbygeneralization}.

\begin{theorem}\label{foxbygeneralization3}
Let $M$ be a finite $R$-module of depth $g$ and dimension $t$. If $\id_RK^i(M)<\infty$ for all $g\leq i<t$, then $$\beta_j(M)=\mu^{j+t}(K(M))$$ for all $j>s+\depth R-t-g$. In particular, $\pd_RM<\infty$ if and only if $\id_RK(M)<\infty$.
\end{theorem}

\begin{proof}
Consider the spectral sequence \ref{eq:2}
$$E_2^{p,q}=\Ext^p_R(k,K^{s-q}(M))\Rightarrow_p\Ext^s_S(\Tor^R_{p+q-s}(k,M),S).$$ The hypothesis and \cite[Theorem 3.7.1]{BH} assures that $E_2^{p,q}=0$ for all $p>\depth R$ and for all $s-t<q\geq s-g$. Therefore, the convergence of $E$ implies that $$\Ext^j_R(k,K(M))\simeq\Ext^s_S(\Tor^R_{j-t}(k,M),S)$$ for all $j>s-\depth R-g$, whence the result.
\end{proof}

The next proposition is an attempt to understand the converse of Corollary \ref{finitepd}.

\begin{proposition}\label{t=g+rfinitepd}
Assume $K^i(M)=0$ for all $i\neq g,t$. If $\pd_RM<\infty$, then $\mu^j(K^g(M))=\mu^{j-g+t+1}(K(M))$ for all $j>\pd_RM+1$.
\end{proposition}

\begin{proof}
The spectral sequence \ref{eq:2} has only two lines as follows
$$\xymatrix@=1em{
0 & 0 & \cdots & 0 & \cdots
\\
\vdots & \vdots &  & \vdots
\\
\Hom_R(k,K^g(M))\ar[ddrrr] & \Ext^1_R(k,K^g(M)) & \cdots & \Ext^{p+r+1}_R(k,K^g(M)) & \cdots
\\
\vdots & \vdots & \ddots & \vdots
\\
\Hom_R(k,K(M)) & \Ext^1_R(k,K(M)) & \cdots & \Ext^{p+r+1}_R(k,K(M)) & \cdots
\\
0 & 0 & \cdots & 0 & \cdots
\\
\vdots & \vdots & & \vdots}$$
Such a shape and convergence yields an exact sequence
$$\xymatrix@=1em{
\Ext^s_R(\Tor^R_{j-g}(k,M),S)\ar[r] & \Ext^j_R(k,K^g(M))\ar[r] & \Ext^{j+r+1}_R(k,K(M))\ar[r] & \Ext^s_S(\Tor^R_{j-g+1}(k,M),S)}$$ for all $j\geq0$. Thus, if $j>\pd_RM+1$, then $$\Ext^j_R(k,K^g(M))\simeq\Ext^{j+r+1}_R(k,K(M))$$ and, in particular, $\mu^j(K^g(M))=\mu^{j+r+1}(K(M))$.
\end{proof}

Corollary \ref{finitepd} and Proposition \ref{t=g+rfinitepd} lead us to ask the following.

\begin{question}\label{question2}
Let $M$ be a finite $R$-module of depth $g$ and dimension $t$. Is it true that $$\pd_RM<\infty \Leftrightarrow \id_RK^i(M)<\infty, \forall i=g,...t?$$
\end{question}

\section*{Acknowledgements}
The authors would like to thank Marc Chardin and Victor Hugo Jorge Pérez for useful comments, suggestions and encouragement. We also thank Thiago Henrique de Freitas for suggesting Corollary \ref{CIchar} and for his valuable comments.

\end{document}